\documentclass[10pt]{amsart}
\usepackage{amsmath,amssymb,amsbsy,amsfonts,latexsym,amsthm,amscd,amsxtra,amsrefs}
\usepackage[T1]{fontenc}
\usepackage{graphics,latexsym,enumerate,pstricks,pst-plot,esint}

\newtheorem{lem}{Lemma}
\newtheorem{thm}{Theorem}
\newtheorem{prop}{Proposition}

\begin{document}

\title[Liouville Systems and Constant Mean Curvature equations]{Wave Equations Associated to Liouville Systems and Constant Mean Curvature equations}
\author{Sagun Chanillo}
\address{Department of Mathematics, Rutgers University, NJ 08854}
\email{chanillo@math.rutgers.edu}
\author{Po-Lam Yung}
\address{Department of Mathematics, Rutgers University, NJ 08854}
\email{pyung@math.rutgers.edu}

\thanks{S.C. is supported by NSF grant DMS-0855541. We would like to thank Carlos Kenig and Andrea Malchiodi for encouragement, interest in this work and useful discussions.}

\maketitle

\section{Introduction} \label{sect:intro}

In this paper we study two types of wave equations, whose elliptic parts exhibit conformal invariance in $\mathbb{R}^2$. 

In the first part, we study the wave equation 
\begin{equation} \label{eq:scalarLiou}
\partial_t^2 u - \Delta_g u = \alpha \left(\frac{e^{2u}}{\fint_{\mathbb{S}^2} e^{2u}} - 1\right)
\end{equation}
on $\mathbb{S}^2$, where $\Delta_g$ denotes the (negative semidefinite) Laplace operator with respect to the standard round metric $g$ on $\mathbb{S}^2$, $\fint_{\mathbb{S}^2}$ is a shorthand for the average $\frac{1}{4\pi} \int_{\mathbb{S}^2} d\text{vol}_g$ with respect to the standard surface measure, and $\alpha$ is some real constant. The stationary (elliptic) analog of this equation is closely connected to the prescribed Gaussian curvature equation in conformal geometry, or Liouville's equation in mathematical physics. In fact, when $\alpha = 1$, if $u$ solves the equation 
\begin{equation} \label{eq:presK}
- \Delta_g u = e^{2u}- 1
\end{equation}
with $\fint_{\mathbb{S}^2} e^{2u} = 1$, then the metric $e^{2u} g$ is another metric on $\mathbb{S}^2$, conformal to $g$, that has Gaussian curvature equal to 1 everywhere, and that has area equal to $4\pi$. Via the stereographic projection that identifies $\mathbb{S}^2$ minus a point with $\mathbb{R}^2$, (\ref{eq:presK}) can also be written as  $-\Delta u = e^{2u}$ on $\mathbb{R}^2$, where now $\Delta$ is the standard Laplacian on $\mathbb{R}^2$. If one replaces $u$ by $u/2 + \log 2$, this becomes Liouville's equation on $\mathbb{R}^2$, namely
\begin{equation} \label{eq:Liou}
-\Delta u = e^u.
\end{equation}
The above stationary equations have been studied by many authors; see for instance the work of Aubin \cite{MR534672}, Chang-Yang \cite{MR925123} and Han \cite{MR1084455} on equation (\ref{eq:presK}), and work of Liouville \cite{MRLiouville}, Chanillo-Kiessling \cite{MR1361515} Chen-Li \cite{MR1121147} and Chou-Wan \cite{MR1262297} on equation (\ref{eq:Liou}).

Our first theorem is the following local existence result:

\begin{thm} \label{thm:localexistence}
Suppose $\alpha \in \mathbb{R}$. For any $u_0 \in \dot{H}^1(\mathbb{S}^2)$ and $u_1 \in L^2(\mathbb{S}^2)$ that satisfy $\int_{\mathbb{S}^2} u_1 = 0$, there exists $u \colon [0,T] \times \mathbb{S}^2 \to \mathbb{R}^N$ in $C^0_t \dot{H}^1_x \cap C^1_t L^2_x$ that solves (\ref{eq:scalarLiou})
with initial data $u(0) = u_0$, $\partial_t u(0)=u_1$, where $T > 0$ depends only on $A$, $\|u_0\|_{\dot{H}^1}$ and $\|u_1\|_{L^2}$. Furthermore, 
$$
\int_{\mathbb{S}^2} u(t) = \int_{\mathbb{S}^2} u_0 
$$
for all $t \in [0,T]$.
\end{thm}

Here and in the sequel, $\int_{\mathbb{S}^2}$ refers to integrals with respect to the standard surface measure $d\text{vol}_g$, and $$\|u_0\|_{\dot{H}^1}^2 = \int_{\mathbb{S}^2} |\nabla u_0|^2, \qquad \|u_1\|_{L^2}^2 = \int_{\mathbb{S}^2} |u_1|^2,$$ where $\nabla$ is the gradient with respect to the standard round metric $g$.

In fact, it also follows from the proof of Theorem~\ref{thm:localexistence} that the initial value problem for (\ref{eq:scalarLiou}) is locally well-posed in $\dot{H}^1 \times L^2$, and that the energy 
$$
E(u(t)) = \fint_{\mathbb{S}^2} \left(|\partial_t u|^2 + |\nabla u|^2 \right) - \alpha \log \left(\fint_{\mathbb{S}^2} e^{2(u-\bar{u})} \right)
$$
of the solution $u(t,x)$ is conserved as $t$ varies as long as the solution exists. 

Next we have the following result for global existence:

\begin{thm} \label{thm:globalexistence}
The solution $u(t,x)$ in Theorem~\ref{thm:localexistence} exists for all time if $\alpha < 1$.
\end{thm}

The main tool in proving this second theorem is the Moser-Trudinger inequality \cite{MR0301504} on the sphere, which says if $u$ is a function on $\mathbb{S}^2$ satisfying $\int_{\mathbb{S}^2} |\nabla u|^2 \leq 1$, then 
$$
\fint_{\mathbb{S}^2} e^{4\pi (u-\bar{u})^2}\leq C
$$
where $\bar{u} = \fint_{\mathbb{S}^2} u$. This inequality is sharp in that one cannot replace $4\pi$ in the exponent by anything that is strictly bigger. Note that this inequality can also be stated as 
$$
\fint_{\mathbb{S}^2} \exp \left( \frac{(u-\bar{u})^2}{\fint_{\mathbb{S}^2} |\nabla u|^2} \right) \leq C,
$$
if $\fint_{\mathbb{S}^2} |\nabla u|^2 < \infty$. 

What we will usually use is the following corollary of the above inequality, namely
\begin{equation} \label{eq:explogMT}
\fint_{\mathbb{S}^2} e^{2(u-\bar{u})}  \leq C \exp\left( \fint_{\mathbb{S}^2} |\nabla u|^2 \right),
\end{equation}
which holds because pointwise $$2(u-\bar{u}) \leq \frac{(u-\bar{u})^2}{\fint_{\mathbb{S}^2} |\nabla u|^2} + \fint_{\mathbb{S}^2} |\nabla u|^2.$$ Equivalently, inequality (\ref{eq:explogMT}) can be stated
\begin{equation} \label{eq:logMT}
\log \left( \fint_{\mathbb{S}^2} e^{2(u-\bar{u})}  \right) \leq \fint_{\mathbb{S}^2} |\nabla u|^2 + \log C.
\end{equation}
%The fact that the constant is $4\pi$ in the Moser-Trudinger inequality is reflected in that the above inequality in that the coefficient of $\fint_{\mathbb{S}^2} |\nabla u|^2$ on the right hand side is 1. 
A celebrated result of Onofri \cite{MR677001} says that the constant $C$ can be taken to be 1 in (\ref{eq:logMT}), but we will not need this in the sequel. 

Using (\ref{eq:logMT}) and conservation of energy, one can then control, as long as the solution exists, the quantity $$\|\partial_t u(t)\|_{L^2(\mathbb{S}^2)} +\|\nabla u(t)\|_{L^2(\mathbb{S}^2)}$$ uniformly in $t$, and this will prove Theorem~\ref{thm:globalexistence}.

In fact Moser \cite{MR0339258} has also proved the following inequality, which says that if $u$ is an \emph{even} function on $\mathbb{S}^2$ satisfying $\int_{\mathbb{S}^2} |\nabla u|^2 \leq 1$, then $$\fint_{\mathbb{S}^2} e^{8\pi (u-\bar{u})^2} \leq C.$$ It follows that for such functions, \begin{equation} \label{eq:logMTeven}
\log \left( \fint_{\mathbb{S}^2} e^{2(u-\bar{u})}  \right) \leq \frac{1}{2} \fint_{\mathbb{S}^2} |\nabla u|^2 + \log C.
\end{equation}
From this, we deduce

\begin{thm} \label{thm:globalexistenceeven}
The solution $u(t,x)$ in Theorem~\ref{thm:localexistence} exists for all time if both $u_0$ and $u_1$ are even functions and if $\alpha < 2$.
\end{thm}

Lin-Zhang \cite{MR2580507} and Chipot-Shafrir-Wolansky \cite{MR1473855} studied the profile of bubbling solutions of the following system of equations on $\mathbb{R}^2$, which was introduced in \cite{MR1361515} and generalizes the Liouville equation (\ref{eq:Liou}):
\begin{equation} \label{eq:Liousys}
-\Delta u_i(x) = \sum_{j=1}^N a_{ij} M_j e^{u_j} \quad i=1, \dots, N.
\end{equation}
Here $(a_{ij})$ is a (constant) $N$ by $N$ symmetric matrix, and $(M_j)$ is a vector. We now turn to a study of the wave analog of this equation, namely
$$
\partial_t^2 u_i - \Delta_g u_i = \sum_{j=1}^N a_{ij} M_j \left(\frac{e^{2u_j}}{\fint_{\mathbb{S}^2} e^{2u_j}} - 1\right), \quad i=1,\dots,N \quad \text{ on $\mathbb{S}^2$}.
$$
This system will be written succintly as 
\begin{equation} \label{eq:waveL}
\partial_t^2 u - \Delta_g u = A \left(\frac{M e^{2u}}{\fint_{\mathbb{S}^2} e^{2u}} - M\right)
\end{equation}
where we think of $u$ as a column vector and write $A$ for the matrix $(a_{ij})$. The bracket on the right hand side is a column vector whose $j$-th entry is $$\frac{M_j e^{2u_j}}{\fint_{\mathbb{S}^2} e^{2u_j}} - M_j.$$ We then have the following result, generalizing Theorem~\ref{thm:localexistence}:

\begin{thm} \label{thm:localexistencesystem}
Suppose $A$, $M$ is as above. For any $u_0 \in \dot{H}^1(\mathbb{S}^2)$ and $u_1 \in L^2(\mathbb{S}^2)$ that takes values in $\mathbb{R}^N$ and satisfy $\int_{\mathbb{S}^2} u_1 = 0$ (i.e. all components of $u_1$ have integral zero), there exists $u \colon [0,T] \times \mathbb{S}^2 \to \mathbb{R}^N$ in $C^0_t \dot{H}^1_x \cap C^1_t L^2_x$ that solves (\ref{eq:waveL})
with initial data $u(0) = u_0$, $\partial_t u(0)=u_1$, where $T > 0$ depends only on $A$, $M$, $\|u_0\|_{\dot{H}^1}$ and $\|u_1\|_{L^2}$. Furthermore, 
$$
\int_{\mathbb{S}^2} u(t) = \int_{\mathbb{S}^2} u_0 
$$
for all $t \in [0,T]$.
\end{thm}

It is also easy to show that (\ref{eq:waveL}) is locally well-posed in $\dot{H}^1 \times L^2$, and that the following energy is conserved over time as long as the solution exists:
$$
E(u(t)) = \fint_{\mathbb{S}^2} \sum_{i,j=1}^N a^{ij} ((\partial_t u_i)(\partial_t u_j) + (\nabla u_i, \nabla u_j)_g) - \sum_{i=1}^N M_i \log \left( \fint_{\mathbb{S}^2}  e^{2(u_i - \bar{u}_i)} \right).
$$ 
Here $(a^{ij})$ is the inverse of the matrix $(a_{ij})$, and $(\cdot, \cdot)_g$ is the inner product of two tangent vectors on $\mathbb{S}^2$ with respect to the metric $g$. 

To study global existence for (\ref{eq:waveL}), we need a generalization of the Moser-Trudinger inequality to systems, which was proved by Shafrir-Wolansky \cite{MR2159222} (see also Wang \cite{MR1689869}). To state this, let
\begin{equation} \label{eq:Lambdef}
\Lambda_J(M) = \sum_{j \in J} M_j - \sum_{i,j \in J} a_{ij} M_i M_j
\end{equation}
for all non-empty subsets $J$ of $\{1, \dots, N\}$. These polynomials in $M$ were first introduced in \cite{MR1361515}, where the symmetry of solutions of (\ref{eq:Liousys}) was studied. Now suppose $A$ is positive definite and has non-negative entries (in addition to being $N$ by $N$ symmetric). Suppose also that $M_j > 0$ for all $j$. Then the generalized Moser-Trudinger inequality says the following: the quantity
\begin{equation} \label{eq:logMTsystem}
\fint_{\mathbb{S}^2} \sum_{i,j=1}^N a^{ij} (\nabla u_i, \nabla u_j)_g - \sum_{i=1}^N M_i \log \left( \fint_{\mathbb{S}^2}  e^{2(u_i - \bar{u}_i)} \right) 
\end{equation}
is bounded below by some finite constant when $u$ varies over all $\mathbb{R}^N$ valued maps in $\dot{H}^1(\mathbb{S}^2)$, if and only if $$\Lambda_J(M) \geq 0$$ for all non-empty subsets $J$ of $\{1,\dots,N\}$. From this and conservation of energy, we deduce the following global existence result:

\begin{thm} \label{thm:globalLsys}
Suppose $A=(a_{ij})$ is a positive definite symmetric $N$ by $N$ matrix with non-negative entries. Suppose $M=(M_j)$ is a column vector in $\mathbb{R}^N$ with positive entries, and suppose $$\Lambda_J(M) > 0$$ for all non-empty subsets $J$ of $\{1, \dots, N\}$. Then the solution $u(t,x)$ in Theorem~\ref{thm:localexistencesystem} exists for all time.
\end{thm}

Next we return to the scalar equation (\ref{eq:scalarLiou}), and study blow up of that equation when $\alpha \geq 1$. An important notion here is the center of mass of the measure $e^{2u} d\text{vol}_g$ for functions $u$ defined on $\mathbb{S}^2$. Given such a function, we define its center of mass to be $$CM(u) = \frac{\int_{\mathbb{S}^2} x e^{2u}}{\int_{\mathbb{S}^2} e^{2u}},$$ where $x$ is the position vector in $\mathbb{R}^3$. Thus $CM(u) \in \mathbb{R}^3$; in fact its length satisfies $|CM(u)| \leq 1$ by the triangle inequality. This center of mass played a crucial role in the work of Chang-Yang \cite{MR925123}  and Han \cite{MR1084455}. There they used the following improved Moser-Trudinger inequality when the center of mass is bounded away from $\mathbb{S}^2$, which was first proved by Aubin \cite{MR534672}. The improved inequality says that if $|CM(u)| \leq 1-\delta$ for some $\delta > 0$, then for any $\mu > 1/2$, there exists a constant $C = C(\mu,\delta)$ such that 
\begin{equation} \label{eq:logMTCM}
\log \left( \fint_{\mathbb{S}^2} e^{2(u-\bar{u})} \right) \leq \mu \fint_{\mathbb{S}^2} |\nabla u|^2 + \log C.
\end{equation} 
One should compare this with (\ref{eq:logMTeven}), since when $u$ is even, $CM(u) = 0$.
 
We have the following blow-up criteria.

\begin{thm} \label{thm:Lblowup}
Let $1 \leq \alpha < 2$. Suppose the solution $u$ in Theorem~\ref{thm:localexistence} exists on a time interval $[0,T_0)$ for some $T_0 < \infty$, and fails to continue beyond $T_0$. Then there is a sequence of times $t_i \to T_0^-$ such that
$$\lim_{i \to \infty} |CM(u,t_i)| = 1,$$
$$\lim_{i \to \infty} \int_{\mathbb{S}^2} e^{2u(t_i)} = \infty,$$
and
$$\lim_{i \to \infty} \|\nabla u(t_i)\|_{L^2} = \infty,$$
where $CM(u,t)$ is the center of mass of $u(t)$. Furthermore, if $\alpha = 1$, then there is some point $p \in \mathbb{S}^2$ such that for any $\varepsilon > 0$, 
$$
\lim_{i \to \infty} \frac{\int_{B(p,\varepsilon)} e^{2u(t_i)} }{\int_{\mathbb{S}^2} e^{2u(t_i)}} \geq 1 - \varepsilon.
$$
Here $B(p,\varepsilon)$ is a geodesic ball on $\mathbb{S}^2$ that is centered at $p$ and of radius $\varepsilon$. 
\end{thm}

In other words, when $\alpha = 1$, if one renormalizes the measures $e^{2u(t_i)} d\text{vol}_g$ so that their integral over $\mathbb{S}^2$ is 1, then the measures concentrates around one single point on $\mathbb{S}^2$ (i.e. there is only one bubble). This is proved using a concentration lemma of Chang and Yang, which we recall in the following section. We do not know whether the same conclusion is true when $\alpha > 1$. 

Finally, we turn to a study of the following system of wave equations on $\mathbb{R}^2$:
\begin{equation} \label{eq:waveCMC}
-\partial_t^2 u + \Delta u = 2 u_x \wedge u_y 
\end{equation}
Here $u$ is a function $u \colon [0,T) \times \mathbb{R}^2 \to \mathbb{R}^3$, $\Delta$ is the Laplacian on $\mathbb{R}^2$ acting componentwise on the three components of $u$, and $u_x \wedge u_y$ is the cross product of the two vectors $u_x$ and $u_y$ in $\mathbb{R}^3$. The stationary analog of this equation is
\begin{equation} \label{eq:CMC}
\Delta u = 2 u_x \wedge u_y.
\end{equation}
This is an interesting equation because if $u$ solves $\Delta u = 2H u_x \wedge u_y$ for some function $H$ on $\mathbb{R}^2$ and satisfies the conformal conditions $|u_x| = |u_y| = 1$ and $u_x \cdot u_y = 0$ everywhere, the the image of $u$ is a surface with mean curvature $H$ in $\mathbb{R}^3$. (\ref{eq:CMC}) is the special case of the above equation when $H \equiv 1$, and is conformally invariant. As a result, we call (\ref{eq:waveCMC}) the wave constant mean curvature (CMC) equation. (\ref{eq:CMC}) is an energy critical equation, in that if $u$ is a solution, then a dilation of $u$ preserving its $\dot{H}^1$ norm is also a solution. Its (entire) solutions in $\dot{H}^1(\mathbb{R}^2)$ were classified by Brezis-Coron; in \cite{MR784102} they showed that if one writes $z$ for the complex coordinate of the domain $\mathbb{R}^2 \simeq \mathbb{C}$ of $u$ and writes $\pi \colon \mathbb{C} \to \mathbb{S}^2 \subseteq \mathbb{R}^3$ for the stereographic projection, then all the solutions of (\ref{eq:CMC}) in $\dot{H}^1$ are of the form $$u(z) = \pi\left(\frac{P(z)}{Q(z)} \right) + C $$ where $P$, $Q$ are polynomials of $z$ and $C$ is a constant vector in $\mathbb{R}^3$. Furthermore, if $u(z)$ is as such, then $$\|\nabla u\|_{L^2}^2 = 8\pi \max\{\text{deg } P, \text{deg } Q\}.$$ It follows that the energy of the (entire) solutions to (\ref{eq:CMC}) are quantized; they are always non-negative integer multiples of $8\pi$. 

Now let $W(z)$ be a ground state solution to (\ref{eq:CMC}); in other words, $W(z)$ is a non-constant solution to (\ref{eq:CMC}) of the form $$W(z) = \pi \left(\frac{P(z)}{Q(z)} \right) + C$$ where $\max\{\text{deg }P, \text{deg }Q\} = 1$ and $\|\nabla W\|_{L^2}^2 = 8\pi$. These will play an important role in our blow up analysis of the wave equation (\ref{eq:waveCMC}). They enter via the following Sobolev inequality. First, it is easy to show, using compensated compactness (aka Wente's inequality) that for all functions $v \in \dot{H}^1(\mathbb{R}^2)$ taking values in $\mathbb{R}^3$, we have 
$$
\left|\int_{\mathbb{R}^2} v \cdot (v_x \wedge v_y) dx dy \right|^{1/3} \leq C \|\nabla v\|_{L^2}.
$$
In fact if $v$ is in $\dot{H}^1(\mathbb{R}^2)$, then $v_x \wedge v_y$ has components in the Hardy space $\mathcal{H}^1(\mathbb{R}^2)$ by compensation compactness, while $v$ itself has components in $BMO$. Thus we have the above inequality. The relevance of $W$ is that the above $W$'s are precisely the minimizers of this inequality; see Caldiroli-Musina \cite{MR2221202}, Lemma 2.1 (and also \cite{MR784102}). We note also that $W$ is a stationary solution to (\ref{eq:waveCMC}), with initial data $u(0) = W$, $\partial_t u(0) = 0$.

The non-linearity occuring on the left hand side of the above Sobolev inequality also arises in the conserved energy of the wave equation (\ref{eq:waveCMC}). In fact if $u(t)$ is a smooth solution to (\ref{eq:waveCMC}) that has compact support on each time slice, then 
$$
E(u(t)) := \int_{\mathbb{R}^2} \frac{1}{2} (|\partial_t u|^2 + |\nabla u|^2) + \frac{2}{3} u \cdot (u_x \wedge u_y) dx dy
$$
is conserved, as one can show by differentiating under the integral. As a result, $E(u(t))$ depends only on the initial data, and it is equal to 
$$
E(u_0,u_1) := \int_{\mathbb{R}^2} \frac{1}{2} (|u_1|^2 + |\nabla u_0|^2) + \frac{2}{3} u_0 \cdot ((u_0)_x \wedge (u_0)_y) dx dy
$$
for all $t$. Our main result is the following:

\begin{thm} \label{thm:finitetimeblowup}
Suppose $u \colon [0,T) \times \mathbb{R}^2 \to \mathbb{R}^3$ is a smooth solution to (\ref{eq:waveCMC}) with initial data $u(0) = u_0$, $u_t(0)=u_1$, and that $u$ has compact support at each time slice $t$. Suppose also that 
$$
E(u_0,u_1) < E(W,0) \quad \text{and} \quad \|\nabla u_0\|_{L^2} > \|\nabla W\|_{L^2}.
$$
Then $T$ is finite; in fact $\|u(t)\|_{L^2(\mathbb{R}^2)}$ cannot remain finite for an infinite amount of time.
\end{thm}

In fact $E(W,0) = 4 \pi/3$ (c.f (\ref{eq:EW}) below) and $\|\nabla W\|_{L^2} = \sqrt{8\pi}$, so the conditions in the above theorem can also be written as $$E(u_0,u_1) < \frac{4 \pi}{3} \quad \text{and} \quad \|\nabla u_0\|_{L^2} > \sqrt{8\pi}.$$ This theorem should be compared to the finite time blow up result of Kenig-Merle \cite{MR2461508} for the energy critical semi-linear focusing wave equation 
$$\partial_t^2 u - \Delta u = |u|^{4/(N-2)} u, \quad \text{on $\mathbb{R} \times \mathbb{R}^N$, $N \geq 3$}.$$

\section{Preliminaries}

Before we move on to the proofs of these theorems, we present some relevant background material. 

First, Aubin \cite{MR534672} proved the following improved Moser-Trudinger inequality:

\begin{prop}[Aubin] \label{lem:Aubin}
Let $f_j \in C^1(\mathbb{S}^2)$, $j = 1, \dots, k$, and $\sum_{j=1}^k |f_j | \geq \delta > 0$ on $\mathbb{S}^2$. Then for any $\mu > 1/2$, there exists $C = C(\mu, \delta, \sum_{j=1}^k \|f_j\|_{C^1(\mathbb{S}^2)})$ such that
$$
\fint_{\mathbb{S}^2} e^{2(u-\bar{u})} \leq C \exp \left(  \mu \fint_{\mathbb{S}^2} |\nabla u|^2 \right)
$$
for all $u \in \dot{H}^1(\mathbb{S}^2)$ satisfying $$\fint_{\mathbb{S}^2} e^{2u} f_j = 0, \quad j = 1, \dots, k.
$$
\end{prop}

See also Lemma 1 of Han \cite{MR1084455}. If one takes $k = 3$ and $$f_j(x) = x_j - \frac{\fint_{\mathbb{S}^2} x_j e^{2u}}{\fint_{\mathbb{S}^2} e^{2u}}, \quad j=1,2,3,$$ then $\sum_{j=1}^3 |f_j|$ is bounded away from zero on $\mathbb{S}^2$ if and only if $|CM(u)|$ is bounded away from 1. This establishes (\ref{eq:logMTCM}) in the Introduction.

Next, Shafrir-Wolansky \cite{MR2159222} proved the following Moser-Trudinger inequality for systems:

\begin{prop}[Shafrir-Wolansky]
Suppose $A=(a_{ij})$ is a positive definite $N$ by $N$ symmetric matrix and has non-negative entries. Suppose also that $M_j > 0$ for all $j$. Then the quantity
\begin{equation} \label{eq:logMTsystemorig}
\frac{1}{2} \int_{\mathbb{S}^2} \sum_{i,j=1}^N a_{ij} (\nabla v_i, \nabla v_j)_g - \sum_{i=1}^N M_i' \log \left( \fint_{\mathbb{S}^2}  \exp \left( \sum_{j=1}^N a_{ij} (v_j - \bar{v}_j) \right) \right) 
\end{equation}
is bounded below by some finite constant when $v$ varies over all $\mathbb{R}^N$ valued maps in $\dot{H}^1(\mathbb{S}^2)$, if and only if 
\begin{equation} \label{eq:Lambdaorig}
 8\pi \sum_{i \in J} M_i' - \sum_{i,j \in J} a_{ij} M_i' M_j' \geq 0
\end{equation}
for all non-empty subsets $J$ of $\{1,\dots,N\}$. 
\end{prop}

See Theorem 2 in \cite{MR2159222}. Now (\ref{eq:logMTsystemorig}) can also be written as
$$
8\pi \left( \fint_{\mathbb{S}^2} \sum_{i,j=1}^N a^{ij} (\nabla u_i, \nabla u_j)_g - \sum_{i=1}^N M_i \log \left( \fint_{\mathbb{S}^2}  e^{2(u_i - \bar{u}_i)} \right) \right)
$$
if we let $2 u_i = \sum_{i=1}^N a_{ij} v_j$ and $M_i = \frac{M_i'}{8\pi}$, and under the same notation, (\ref{eq:Lambdaorig}) is equivalent to $$\Lambda_J (M) \geq 0$$ where $\Lambda_J$ is defined as in (\ref{eq:Lambdef}). Thus we recover the generalized Moser-Trudinger inequality (\ref{eq:logMTsystem}) stated in the Introduction. From this we deduce the following:

\begin{prop} \label{prop:logMTsystempositive}
Suppose $A$ is as in the previous Proposition, and $M_j > 0$ for all $j$. Suppose also that $\Lambda_J(M) > 0$ for all non-empty subsets $J$ of $\{1, \dots, N\}$. Then there are some constants $\varepsilon > 0$ and $C$ (both depending only on $A$ and $M$) such that  
\begin{equation} \label{eq:logMTsystempositive}
\fint_{\mathbb{S}^2} \sum_{i,j=1}^N a^{ij} (\nabla u_i, \nabla u_j)_g - \sum_{i=1}^N M_i \log \left( \fint_{\mathbb{S}^2}  e^{2(u_i - \bar{u}_i)} \right) \geq \varepsilon \fint_{\mathbb{S}^2} |\nabla u|^2 - C
\end{equation}
for all $u \in \dot{H}^1(\mathbb{S}^2)$ taking values in $\mathbb{R}^N$.
\end{prop}

In fact, the condition $\Lambda_J(M) > 0$ for all non-empty $J \subset \{1, \dots, N\}$ is an open condition. Thus one is led to define $A' = A - 2\varepsilon Id$ where $Id$ is the identity matrix and $\varepsilon > 0$ is some small constant. Suppose $\varepsilon$ is sufficiently small. Then $A'$ is still positive definite symmetric with non-negative entries, and if $\Lambda'_J(M)$ is defined in the same way as $\Lambda_J(M)$, except one replaces entries of $A$ by the corresponding entries of $A'$, then one still has $\Lambda'_J(M) > 0$ for all non-empty $J \subseteq \{1, \dots, N\}$. Thus by the previous assertion,
$$
\fint_{\mathbb{S}^2} \sum_{i,j=1}^N (a')^{ij} (\nabla u_i, \nabla u_j)_g - \sum_{i=1}^N M_i \log \left( \fint_{\mathbb{S}^2}  e^{2(u_i - \bar{u}_i)} \right) \geq - C
$$
where $(a')^{ij}$ are the entries of $A'$. Now $A^{-1} = (A')^{-1} + 2\varepsilon Id + O(\varepsilon^2)$. Thus if $\varepsilon$ is sufficiently small, (\ref{eq:logMTsystempositive}) follows.

We also need the following concentration lemma of Chang-Yang \cite{MR925123} (see Proposition A there). Let $$S[u] =  \fint_{\mathbb{S}^2} |\nabla u|^2 + 2 \fint_{\mathbb{S}^2} u $$ for any function $u \in \dot{H}^1(\mathbb{S}^2)$.

\begin{prop}[Chang-Yang] \label{lem:CY}
Suppose $u_j \in \dot{H}^1(\mathbb{S}^2)$ a sequence of functions with $\fint_{\mathbb{S}^2} e^{2u_j} = 1$ and $\sup_j S[u_j] = C < \infty,$ we have either $$\sup_j \fint_{\mathbb{S}^2} |\nabla u_j|^2 = C' < \infty,$$ or there exists a point $p \in \mathbb{S}^2$ and a subsequence of $u_j$ (which we still denote by $u_j$) such that for any $\varepsilon > 0$, we have $$\lim_{j \to \infty}  \frac{1}{4\pi} \int_{B(p,\varepsilon)} e^{2u_j} \geq 1-\varepsilon.$$ Here $B(p,\varepsilon)$ is the geodesic ball on $\mathbb{S}^2$ centered at $p$ and of radius $\varepsilon$.
\end{prop}

Finally, we need the following elementary result about the solution of wave equations on the sphere. To state this, first recall that every $L^2$ function on $\mathbb{S}^2$ can be decomposed as a convergent sum of eigenfunctions of the Laplacian $\Delta_g$ on $\mathbb{S}^2$. Using this, one can define the spectral multiplier $\cos(\sqrt{-\Delta_g})$ on $L^2$ functions on $\mathbb{S}^2$, as well as $\frac{\sin(\sqrt{-\Delta_g})}{\sqrt{-\Delta_g}}$ on $L^2$ functions on $\mathbb{S}^2$ whose integral is zero. These operators solve the initial value problem 
\begin{equation} \label{eq:wave}
\begin{cases}
\partial_t^2 v - \Delta_g v = f \\
v(0) = u_0, \quad \partial_t v(0) = u_1
\end{cases}
\end{equation}
on $[0,\infty) \times \mathbb{S}^2$ via the following Duhamel formula:

\begin{prop} \label{prop:wave}
If $u_0 \in \dot{H}^1(\mathbb{S}^2)$, $u_1 \in L^2(\mathbb{S}^2)$, $f \in L^1_t L^2_x ([0,T) \times \mathbb{S}^2)$ with $$\int_{\mathbb{S}^2} u_1 = 0 = \int_{\mathbb{S}^2} f(s) \quad \text{for all $s \in [0,T)$},$$ then (\ref{eq:wave}) has a unique solution $v(t,x) \in C^0_t \dot{H}^1_x \cap C^1_t L^2_x ([0,T) \times \mathbb{S}^2)$ given by 
$$
v(t,x) = \cos(t\sqrt{-\Delta_g})u_0 + \frac{\sin(t\sqrt{-\Delta_g})}{\sqrt{-\Delta_g}} u_1 + \int_0^t \frac{\sin((t-s)\sqrt{-\Delta_g})}{\sqrt{-\Delta_g}} f(s) ds.
$$
Furthermore, the solution $v(t,x)$ satisfies
$$
\|v\|_{C^0_t \dot{H}^1_x} + \|\partial_t v\|_{C^0_t L^2_x} \leq 2 (\|u_0\|_{\dot{H}^1} + \|u_1\|_{L^2} + \|f\|_{L^1_t L^2_x}).
$$
All these hold as well if $u_0$, $u_1$ and $f(s,\cdot)$ all takes value in $\mathbb{R}^N$ for some $N$.
\end{prop}

\section{Liouville's equation}

It is clear that Theorem~\ref{thm:localexistence} follows from Theorem~\ref{thm:localexistencesystem}. So it suffices to prove Theorem~\ref{thm:localexistencesystem}.

\begin{proof}[Proof of Theorem~\ref{thm:localexistencesystem}]
The proof proceeds by a standard fixed point argument. Suppose the initial data $u_0$ and $u_1$ are given as above. Let 
$$
R = 3(\|u_0\|_{\dot{H}^1} + \|u_1\|_{L^2}),
$$
and
$$
I = \fint_{\mathbb{S}^2} u_{0}.
$$
Let $B_T$ be the collection of all $\mathbb{R}^N$-valued functions $u \in C^0\dot{H}^1 \cap C^1L^2 ([0,T] \times \mathbb{S}^2)$ satisfying $$\fint_{\mathbb{S}^2} u(s) = I \quad \text{for all $s \in [0,T]$},$$ and $$\|u\|_{C^0\dot{H}^1([0,T] \times \mathbb{S}^2)} + \|\partial_t u\|_{C^0L^2([0,T] \times \mathbb{S}^2)} \leq R.$$ For $u \in B_T$, one can solve the initial value problem
$$
\begin{cases}
\partial_t^2 v - \Delta_g v = A \left(\frac{M e^{2u}}{\fint_{\mathbb{S}^2} e^{2u}} - M\right) \\
v(0) = u_0, \quad \partial_t v(0) = u_1
\end{cases}
$$
on $[0,T] \times \mathbb{S}^2$, since both $u_1$ and the right hand side of the first equation have integral zero on $\mathbb{S}^2$. Now we claim $v \in B_T$ if $T$ is sufficiently small: in fact by Proposition~\ref{prop:wave},
\begin{align*}
&\|v\|_{C^0\dot{H}^1([0,T] \times \mathbb{S}^2)} + \|\partial_t v\|_{C^0L^2([0,T] \times \mathbb{S}^2)} \\
\leq & 2\|u_0\|_{\dot{H}^1} + 2\|u_1\|_{L^2} + 
2\beta \int_0^T \|\left(\frac{M e^{2u}}{\fint_{\mathbb{S}^2} e^{2u}} - M\right)(s)\|_{L^2} ds
\end{align*}
where $\beta$ is the operator norm of $A$. The last integral is bounded by $$CT|M| + \sum_{j=1}^N CT|M|e^{-2I_j} \max_{s \in [0,T]} \|e^{2u_j(s)}\|_{L^2}$$
where $I_j = \bar{u}_j$ is the $j$-th component of $I$, since $$\fint_{\mathbb{S}^2} e^{2u_j} \geq \exp\left(\fint_{\mathbb{S}^2} 2u_j \right) = e^{2I_j}.$$
But for $s \in [0,T]$, 
\begin{align}
\frac{1}{4\pi} \|e^{2u_j(s)}\|_{L^2}^2 
&= \fint_{\mathbb{S}^2} e^{4u_j(s)} \notag \\ 
&= \fint_{\mathbb{S}^2} e^{4(u_j-\bar{u_j})(s)} e^{4 I_j} \notag \\
&\leq C \exp\left(\fint_{\mathbb{S}^2}  |\nabla (2 u_j)|^2(s) \right) e^{4 I_j} \notag \\
&\leq C e^{4 R^2} e^{4I_j}. \label{eq:euL2}
\end{align}
The second to last inequality follows from the Moser-Trudinger inequality on the sphere.
Hence 
\begin{align*}
&\|v\|_{C^0\dot{H}^1([0,T] \times \mathbb{S}^2)} + \|\partial_t v\|_{C^0L^2([0,T] \times \mathbb{S}^2)} \\
\leq & 2\|u_0\|_{\dot{H}^1} + 2\|u_1\|_{L^2} + 2C\beta T|M| + 2C\beta T |M| e^{2R^2}.
\end{align*}
If $T$ is sufficiently small, depending only on $\beta$, $|M|$ and the norms of the initial data, then the above is bounded by $R$. Also, $\fint_{\mathbb{S}^2} v(s) = I$ for all $s \in [0,T]$. Thus $v \in B_T$ if $T$ is as such. This defines a map $F \colon B_T \to B_T$ given by $u \mapsto v$. We show that by further shrinking $T$ if necessary, this map is a contraction.

Suppose $u^{(1)}$, $u^{(2)}$ are in $B_T$, and $v^{(1)} = F(u^{(1)})$, $v^{(2)} = F(u^{(2)})$. Then $v := v^{(1)} - v^{(2)}$ satisfies
$$
\begin{cases}
\partial_t^2 v - \Delta v = A \left(\frac{M e^{2u^{(1)}}}{\fint_{\mathbb{S}^2} e^{2u^{(1)}} } - \frac{M e^{2u^{(2)}}}{\fint_{\mathbb{S}^2} e^{2u^{(2)}} }\right) \\
v(0) = 0, \quad \partial_t v(0) = 0.
\end{cases}
$$
Hence 
\begin{align*}
&\|v\|_{C^0\dot{H}^1([0,T] \times \mathbb{S}^2)} + \|\partial_t v\|_{C^0L^2([0,T] \times \mathbb{S}^2)} \\
\leq & 
2\beta \int_0^T \left\|\left(\frac{M e^{2u^{(1)}}}{\fint_{\mathbb{S}^2} e^{2u^{(1)}}} - \frac{M e^{2u^{(2)}}}{\fint_{\mathbb{S}^2} e^{2u^{(2)}} }\right)(s)\right\|_{L^2} ds.
\end{align*}
But for $s \in [0,T]$ and $j =1, \dots, N$,
\begin{align*}
&\left\|\left(\frac{e^{2u^{(1)}_j}}{\fint_{\mathbb{S}^2} 2e^{u^{(1)}_j}} - \frac{e^{2u^{(2)}_j}}{\fint_{\mathbb{S}^2} e^{2u^{(2)}_j} }\right)(s)\right\|_{L^2} \\
\leq &
\left\| \frac{e^{2u^{(1)}_j}-e^{2u^{(2)}_j}}{\fint_{\mathbb{S}^2} e^{2u^{(1)}_j} } \right\|_{L^2}(s) + \left\| \frac{e^{2u^{(2)}_j} \left(\fint_{\mathbb{S}^2} e^{2u^{(1)}_j}  - \fint_{\mathbb{S}^2} e^{2u^{(2)}_j} \right) }{\left(\fint_{\mathbb{S}^2} e^{2u^{(1)}_j} \right) \left( \fint_{\mathbb{S}^2} e^{2u^{(2)}_j}  \right) }\right\|_{L^2}(s).
\end{align*}
The first term is bounded by
\begin{align*}
& e^{-2I_j} \|(2u^{(1)}_j(s)-2u^{(2)}_j(s))(e^{2u^{(1)}_j(s)} + e^{2u^{(2)}_j(s)})\|_{L^2} \\
\leq & e^{-2I_j} \|2u^{(1)}_j(s)-2u^{(2)}_j(s)\|_{L^4} \left(\|e^{2u^{(1)}_j(s)}\|_{L^4} + \|e^{2u^{(2)}_j(s)}\|_{L^4}\right).
\end{align*}
Using Sobolev inequality for the first factor (since $\int_{\mathbb{S}^2} (u^{(1)}_j(s)-u^{(2)}_j(s)) = 0$ for all $s$) and use
$$
\frac{1}{4\pi} \|e^{2u^{(k)}_j(s)}\|_{L^4}^4 = \fint_{\mathbb{S}^2} e^{8(u^{(k)}_j(s)-I_j)} e^{8I_j} \leq C \exp\left(\fint_{\mathbb{S}^2} | \nabla (4u^{(k)}_j)|^2 (s) \right) e^{8 I_j}
\leq C e^{16 R^2} e^{8 I_j}
$$
for the second factor, we bound this by
$$
C e^{4 R^2} \|u^{(1)}(s)-u^{(2)}(s)\|_{\dot{H}^1}.
$$
For the second term, we bound that by
\begin{align*}
& C e^{-2I_j} e^{-2I_j} \|e^{2u^{(2)}_j(s)}\|_{L^2} \int_{\mathbb{S}^2} |2u^{(1)}_j(s)-2u^{(2)}_j(s)| (e^{2u^{(1)}_j(s)} + e^{2u^{(2)}_j(s)}) \\
\leq & C e^{-2I_j} e^{-2I_j} \|e^{2u^{(2)}_j(s)}\|_{L^2} (\|e^{2u^{(1)}_j(s)}\|_{L^2} + \|e^{2u^{(2)}_j(s)}\|_{L^2}) \|2u^{(1)}_j(s)-2u^{(2)}_j(s)\|_{L^2} \\
\leq & C e^{4R^2} \|u^{(1)}(s)-u^{(2)}(s)\|_{\dot{H}^1} 
\end{align*}
since 
$$\|e^{2u^{(k)}_j(s)}\|_{L^2} \leq C e^{2R^2} e^{2I_j}$$
by (\ref{eq:euL2}), and we can use Sobolev inequality for the last term. It follows that
$$
\|v\|_{C^0\dot{H}^1([0,T] \times \mathbb{S}^2)} + \|\partial_t v\|_{C^0L^2([0,T] \times \mathbb{S}^2)} \leq 
2 C\beta T |M| e^{4R^2} \|u^{(1)}-u^{(2)}\|_{C^0\dot{H}^1([0,T] \times \mathbb{S}^2)}.
$$
Hence if $T$ is sufficiently small with respect to $\beta$, $|M|$ and the norms of the initial data, then the map $F$ we defined above is a contraction map. The contraction mapping principle then says that $F$ has a fixed point, which gives the desired solution to the initial value problem in the theorem.
\end{proof}

It is now a standard matter to modify the above proof and show that the initial value problem in Theorem~\ref{thm:localexistence} or Theorem~\ref{thm:localexistencesystem} is locally well-posed, whose detail we omit. Thus to check conservation of energy for the solution $u$ in Theorem~\ref{thm:localexistencesystem}, we may assume without loss of generality that $u$ is smooth, in which case one can differentiate
%\begin{align*}
%E(u(t)) &= \fint_{\mathbb{S}^2} \left(|\partial_t u|^2 + |\nabla u|^2 \right) - \alpha \log \left(\fint_{\mathbb{S}^2} e^{2(u-\bar{u})} \right)\\
%&=  \fint_{\mathbb{S}^2} \left(|\partial_t u|^2 + |\nabla u|^2 \right) - \alpha \left( \log \left(\fint_{\mathbb{S}^2} e^{2u} \right) -  2 \bar{u} \right)
%\end{align*}
%and get
%\begin{align*}
%\partial_t E(u(t)) &= \fint_{\mathbb{S}^2} (2 \partial_t u \cdot \partial_t^2 u + 2  \nabla (\partial_t u) \cdot \nabla u) - \alpha \frac{ \fint_{\mathbb{S}^2} 2(\partial_t u) e^{2u} }{\fint_{\mathbb{S}^2} e^{2u}}. 
%\end{align*}
%The first term is equal to
%$$
%\fint_{\mathbb{S}^2} 2 \partial_t u (\partial_t^2 u - \Delta_g u) 
%= \alpha \fint_{\mathbb{S}^2} 2\partial_t u \left( \frac{e^{2u} }{\fint_{\mathbb{S}^2} e^{2u}} \right)
%$$
%by equation (\ref{eq:scalarLiou}), since $$\fint_{\mathbb{S}^2} 2\partial_t u \cdot 1 = 2 \partial_t \fint_{\mathbb{S}^2} u = 0$$ by Theorem~\ref{thm:localexistence}. This cancels with the second term, and thus $\partial_t E(u(t)) = 0$. 
%
%A similar argument proves conservation of energy for the system (\ref{eq:waveL}): in fact then the energy is 
\begin{align*}
E(u(t)) &= \fint_{\mathbb{S}^2} \sum_{i,j=1}^N a^{ij} ((\partial_t u_i)(\partial_t u_j) + (\nabla u_i, \nabla u_j)_g) - \sum_{i=1}^N M_i \log \left( \fint_{\mathbb{S}^2}  e^{2(u_i-\bar{u}_i)} \right) \\
&=\fint_{\mathbb{S}^2} \sum_{i,j=1}^N a^{ij} ((\partial_t u_i)(\partial_t u_j) + (\nabla u_i, \nabla u_j)_g) - \sum_{i=1}^N M_i \log \left( \fint_{\mathbb{S}^2}  e^{2u_i} \right) + \sum_{i=1}^N 2 M_i \bar{u}_i
\end{align*}
and get
\begin{align*}
\partial_t E(u(t)) 
= \fint_{\mathbb{S}^2} \sum_{i,j=1}^N 2 a^{ij} ((\partial_t u_i)(\partial_t^2 u_j) + (\nabla \partial_t u_i, \nabla u_j)_g) -  \sum_{i=1}^N M_i \frac{\fint_{\mathbb{S}^2} 2 (\partial_t u_i) e^{2u_i}}{\fint_{\mathbb{S}^2}  e^{2u_i}}.
\end{align*}
The first integral is equal to
\begin{align*}
&\fint_{\mathbb{S}^2} \sum_{i,j=1}^N 2 a^{ij} (\partial_t u_i)(\partial_t^2 u_j - \Delta_g u_j) \\
=& \fint_{\mathbb{S}^2} \sum_{i,j=1}^N 2 a^{ij} (\partial_t u_i) \sum_{k = 1}^N a_{jk} M_k \left( \frac{e^{2u_k}}{\fint_{\mathbb{S}^2} e^{2u_k}} - 1 \right)\\
=& \fint_{\mathbb{S}^2} \sum_{i=1}^N 2 (\partial_t u_i) M_i \left( \frac{e^{2u_i}}{\fint_{\mathbb{S}^2} e^{2u_i}} - 1 \right)\\
=& \fint_{\mathbb{S}^2} \sum_{i=1}^N 2 (\partial_t u_i) M_i \frac{e^{2u_i}}{\fint_{\mathbb{S}^2} e^{2u_i}}\\
\end{align*}
by equation (\ref{eq:waveL}), since $$\fint_{\mathbb{S}^2} 2(\partial_t u_i) M_i = 2 M_i \partial_t \fint_{\mathbb{S}^2} u_i = 0$$ by Theorem~\ref{thm:localexistencesystem}. This cancels with the second term, and thus $\partial_t E(u(t)) = 0$. Similarly one proves conservation of energy for the solution $u$ in Theorem~\ref{thm:localexistence}.

\begin{proof}[Proof of Theorem~\ref{thm:globalexistence}]
Fix $\alpha < 1$. We only need to show that if for some $T > 0$, $u \in C^0\dot{H}^1 \cap C^1L^2 ([0,T) \times \mathbb{S}^2)$ satisfy
$$
\partial_t^2 u - \Delta u = \alpha \left(\frac{e^{2u}}{\fint_{\mathbb{S}^2} e^{2u} } - 1 \right) \quad \text{on $[0,T) \times \mathbb{S}^2$}
$$
with initial data $u(0) = u_0$, $u_t(0) = u_1$, then
$$
\|u\|_{C^0\dot{H}^1([0,T) \times \mathbb{S}^2)} + \|\partial_t u\|_{C^0L^2([0,T) \times \mathbb{S}^2)} \leq B
$$
for some number $B$ that depends only on $\alpha$, $\|u_0\|_{\dot{H}^1}$ and $\|u_1\|_{L^2}$ (because if this is true, then one can extend the solution for a fixed amount of time beginning at any $t \in [0,T)$, which in particular says that the solution extends beyond time $T$).

Now we need only conservation of energy and the Moser-Trudinger inequality (more precisely, its corollary as in (\ref{eq:logMT})) to accomplish this. Recall that
$$
E(u(t)) = \fint_{\mathbb{S}^2} \left(|\partial_t u|^2 + |\nabla u|^2 \right) - \alpha \log \left(\fint_{\mathbb{S}^2} e^{2(u-\bar{u})} \right)
$$
is conserved over time. But (\ref{eq:logMT}) implies
$$
\log \left(\fint_{\mathbb{S}^2} e^{2(u_j-\bar{u}_j)} \right) \leq \fint_{\mathbb{S}^2} |\nabla u_j|^2 + \log C.
$$
at any time $t$. Hence if $\alpha \in [0,1)$, then at every time $t \in [0,T)$, we have
$$
\fint_{\mathbb{S}^2} \left( |\partial_t u(t)|^2 + (1-\alpha) |\nabla u(t)|^2 \right) \leq  E(u(t)) + \alpha \log C = E(u(0)) + \alpha \log C,
$$
the last inequality following from conservation of energy. The left-hand side is bounded below by a constant times $\|\partial_t u(t)\|_{L^2}^2 + \|\nabla u(t)\|_{L^2}^2$, so
$$
\|u\|_{C^0\dot{H}^1([0,T) \times \mathbb{S}^2)} + \|\partial_t u\|_{C^0L^2([0,T) \times \mathbb{S}^2)} \leq \frac{E(u(0)) + \alpha \log C}{1-\alpha}.
$$
On the other hand, if $\alpha < 0$, since Jensen's inequality implies
$$
\log \left(\fint_{\mathbb{S}^2} e^{2(u-\bar{u})} \right) \geq 0,
$$
we have
$$
\|u\|_{C^0\dot{H}^1([0,T) \times \mathbb{S}^2)} + \|\partial_t u\|_{C^0L^2([0,T) \times \mathbb{S}^2)} \leq E(u(0)).
$$ 
In both cases, $\|u\|_{C^0\dot{H}^1([0,T) \times \mathbb{S}^2)} + \|\partial_t u\|_{C^0L^2([0,T) \times \mathbb{S}^2)}$ is bounded by a constant that depends only on $\alpha$ and the norms of the initial data. This completes the proof.
\end{proof}

\begin{proof}[Proof of Theorem~\ref{thm:globalexistenceeven}]
The proof of Theorem~\ref{thm:globalexistenceeven} is the same as that of Theorem~\ref{thm:globalexistence}, except one uses the improved Moser-Trudinger inequality (\ref{eq:logMTeven}) in place of (\ref{eq:logMT}), which is possible since if $u_0$ and $u_1$ are even, then $u(t)$ remains even as long as the solution exists. We omit the details.
\end{proof}

\begin{proof}[Proof of Theorem~\ref{thm:globalLsys}]
This time we need the Moser-Trudinger inequality for systems apart from conservation of energy. Suppose $A$ and $M$ are as in the statement of the theorem. We only need to show that if for some $T > 0$, $u \in C^0\dot{H}^1 \cap C^1L^2 ([0,T) \times \mathbb{S}^2)$ satisfy
$$
\partial_t^2 u - \Delta u = A \left(\frac{Me^{2u}}{\fint_{\mathbb{S}^2} e^{2u} } - M \right) \quad \text{on $[0,T) \times \mathbb{S}^2$}
$$
with initial data $u(0) = u_0$, $u_t(0) = u_1$, then
$$
\|u\|_{C^0\dot{H}^1([0,T) \times \mathbb{S}^2)} + \|\partial_t u\|_{C^0L^2([0,T) \times \mathbb{S}^2)} \leq B
$$
for some number $B$ that depends only on $A$, $M$, $\|u_0\|_{\dot{H}^1}$ and $\|u_1\|_{L^2}$. Now conservation of energy says that 
$$
E(u(t)) = \fint_{\mathbb{S}^2} \sum_{i,j=1}^N a^{ij} ((\partial_t u_i)(\partial_t u_j) + (\nabla u_i, \nabla u_j)_g) - \sum_{i=1}^N M_i \log \left( \fint_{\mathbb{S}^2}  e^{2(u_i-\bar{u}_i)} \right) 
$$
is conserved over time. But by Proposition~\ref{prop:logMTsystempositive}, there exists some $\varepsilon > 0$ and $C$, depending only on $A$ and $M$, such that at any time $t \in [0,T)$,
$$
\fint_{\mathbb{S}^2} \sum_{i,j=1}^N a^{ij} (\nabla u_i, \nabla u_j)_g - \sum_{i=1}^N M_i \log \left( \fint_{\mathbb{S}^2}  e^{2(u_i-\bar{u}_i)} \right) 
\geq \varepsilon \fint_{\mathbb{S}^2} |\nabla u|^2 - C.
$$
Since $A$ (and hence $A^{-1}$) is positive definite, this implies $$E(u(t)) \geq \varepsilon' \fint_{\mathbb{S}^2} (|\partial_t u|^2 + |\nabla u|^2) - C$$
for some $\varepsilon' > 0$. By conservation of energy, this proves that $$\|\partial_t u(t)\|_{L^2(\mathbb{S}^2)} + \|\nabla u(t)\|_{L^2(\mathbb{S}^2)}$$ is bounded uniformly in $t$ by some constant that depends only on $A$, $M$ and the norms of the initial data. This completes our proof. 
\end{proof}

\begin{proof}[Proof of Theorem~\ref{thm:Lblowup}]
Suppose $1 \leq \alpha < 2$. First, if the solution $u$ in Theorem~\ref{thm:localexistence} fails to exist globally in time, say $u$ only exists on the time interval $[0,T)$ where $T > 0$ is finite, then 
\begin{equation} \label{eq:CMonsphere}
\limsup_{t \to T} |CM(t)| = 1,
\end{equation} where $CM(t)$ is the center of mass of $u$ at time $t$. This is because otherwise one can find a sequence of times $t_i < T$, $t_i \to T$ such that $\lim_{i \to \infty} |CM(t_i)| < 1,$ which implies via Aubin's result that equation (\ref{eq:logMTCM}) holds along this sequence of times $t_i$ for some $\frac{1}{2} < \mu < \frac{1}{\alpha}$ and some finite constant $C$. (Such $\mu$ exists since $1 \leq \alpha < 2$.) Now since the energy 
$$
E(u(t_i)) = \fint_{\mathbb{S}^2} (|\partial_t u(t_i)|^2 + |\nabla u(t_i)|^2) - \alpha \log \left( \fint_{\mathbb{S}^2} e^{2(u(t_i)-\bar{u}(t_i))} \right) 
$$
is independent of $t_i$, this shows that
$$
\fint_{\mathbb{S}^2} \left(|\partial_t u(t_i)|^2 + (1-\alpha \mu)|\nabla u(t_i)|^2\right) 
$$
has a uniform upper bound independent of $t_i$. Since $1-\alpha\mu > 0$, the same holds for $\|\partial_t u(t_i)\|_{L^2} + \|\nabla u(t_i)\|_{L^2}$, and by Theorem~\ref{thm:localexistence} this proves that the solution extends past $t_i$ for a fixed amount of time for all $i$. This contradicts that $T$ is the maximal time of existence of the solution, and this proves (\ref{eq:CMonsphere}).

Next, fix for the moment on a sequence $t_i \to T$ from below such that $$\lim_{i \to \infty} |CM(t_i)| = 1.$$ We claim that by passing to a subsequence, which we still denote as $t_i$, we have $$\lim_{i \to \infty} \int_{S^2} e^{2u(t_i)} = \infty.$$ This is because otherwise we have $$\limsup_{i \to \infty} (\|\partial_t u(t_i)\|_{L^2} + \|\nabla u(t_i)\|_{L^2}) \lesssim E(u(0)) + \alpha \limsup_{i \to \infty} \fint_{\mathbb{S}^2} e^{2u(t_i)} + \alpha \bar{u}(0) < \infty$$ by conservation of energy and conservation of $\bar{u}$, so one can extend the solution beyond $T$, contradicting the maximality of $T$. Thus we have the second assertion of the theorem. From this and the Moser-Trudinger inequality (\ref{eq:logMT}) we obtain the third assertion, namely $$\lim_{i \to \infty} \|\nabla u(t_i)\|_{L^2} = \infty$$ since $\bar{u}$ is constant.

Finally, suppose $\alpha = 1$, and $t_i$ is the above chosen subsequence. Let $$m_i = \fint_{S^2} e^{2u(t_i)}$$ and let $v_i(x) = u(x,t_i) - \frac{1}{2} \log m_i$. We will apply Chang-Yang's concentration lemma (Proposition~\ref{lem:CY}) to this sequence of functions $v_i$. First we observe that $$\fint_{S^2} e^{2v_i} = 1$$ for all $i$ by definition of $v_i$. Next we check that $$\limsup_{i \to \infty} \|\nabla v_i\|_{L^2} = \infty.$$ But since $u(t_i)$ and $v_i$ differ by only a constant, this is true by the analogous property of $u(t_i)$. Finally we check that $S[v_i] \leq C$ uniformly in $i$. But $S[v_i] = \fint_{S^2} |\nabla v_i|^2 + 2 \fint_{S^2} v_i$,  $\nabla v_i = \nabla u(x,t_i)$, and $\fint_{S^2} v_i = \fint_{S^2} u(x,t_i) - \frac{1}{2} \log m_i.$ It follows that 
\begin{align*}
S[v_i] &= \fint_{S^2} |\nabla u(x,t_i)|^2 dx - \log \left(\fint_{S^2} e^{2u(t_i)-\bar{u}(t_i)} \right) \\
&= E(u(t_i)) - \fint_{\mathbb{S}^2} |\partial_t u(t_i)|^2 \\
&\leq E(u(0)).
\end{align*}
The second equality holds because now $\alpha = 1$. As a result, Chang-Yang's concentration lemma applies. But we already knew previously that $\lim_{i \to \infty} \|\nabla v_i\|_{L^2} = \lim_{i \to \infty} \|\nabla u(t_i)\|_{L^2} = \infty$. Thus we get the existence of some $p \in \mathbb{S}^2$ and a subsequence of $t_i$ (which we still denote by $t_i$) such that for any $\varepsilon > 0$, $$\lim_{i \to \infty} \frac{1}{4\pi} \int_{B(p,\varepsilon)} e^{2v_i} \geq 1-\varepsilon.$$ Writing out the definition of $v_i$, we get the last assertion in Theorem~\ref{thm:Lblowup}.
\end{proof}

\section{Wave CMC equation}

We now turn to the wave CMC equation (\ref{eq:waveCMC}).

\subsection{Time-independent variational estimates}

In this subsection $u$ will be a map from $\mathbb{R}^2$ into $\mathbb{R}^3$ independent of time. The script $\mathcal{E}$ will be the stationary energy, i.e. 
\begin{equation} \label{eq:Sob}
\mathcal{E}(u) := \int_{\mathbb{R}^2} \frac{1}{2} |\nabla u|^2 + \frac{2}{3} u \cdot (u_x \wedge u_y) dx,
\end{equation}
if $u$ is a $\dot{H}^{1}$ function on $\mathbb{R}^2$ taking values in $\mathbb{R}^3$. Here we wrote $dx$ for the Lebesgue measure on $\mathbb{R}^2$ (instead of $dxdy$), and we will adapt this notation throughout. We have the following Sobolev inequality:

\begin{lem}
If $u$ is a $\dot{H}^{1}$ function on $\mathbb{R}^2$ taking values in $\mathbb{R}^3$, then 
$$
\left| \int_{\mathbb{R}^2} u \cdot (u_x \wedge u_y) dx \right| \leq C \|\nabla u\|_{L^2}^3.
$$ 
Furthermore, any ground state solution $W$ to (\ref{eq:CMC}) (as described in Section~\ref{sect:intro}) is a minimizer of this Sobolev inequality.
\end{lem}

See \cite{MR784102} and \cite{MR2221202}. From now on we write $C$ for the best constant of the above inequality. Then
\begin{equation} \label{eq:West1}
\left|\int_{\mathbb{R}^2} W \cdot (W_x \wedge W_y) dx \right| = C \|\nabla W\|_{L^2}^3. 
\end{equation}
Also, from the equation 
$$
\Delta W = 2 W_x \wedge W_y,
$$
one easily deduces (by multiplying both sides by $W$ and integrating by parts) that
\begin{equation} \label{eq:West2}
\int_{\mathbb{R}^2} W \cdot (W_x \wedge W_y) dx = -\frac{1}{2} \|\nabla W\|_{L^2}^2.
\end{equation}
Together with (\ref{eq:West1}) we see that 
\begin{equation} \label{eq:Sobc}
C = \frac{1}{2 \|\nabla W\|_{L^2}}.
\end{equation}
Now let 
$$
f(\lambda) = \frac{1}{2} \lambda^2 - \frac{2}{3} C \lambda^3 \quad \text{for $\lambda > 0$}.
$$
Then the critical points of $f$ are $\lambda = 0$ and $(2C)^{-1} = \|\nabla W\|_{L^2}$. The function value of $f$ at the critical points are $f(0) = 0$ and 
\begin{align*}
f(\|\nabla W\|_{L^2}) 
&= \frac{1}{2} \|\nabla W\|_{L^2}^2 - \frac{2}{3} C \|\nabla W\|_{L^2}^3 \\
&= \frac{1}{2} \|\nabla W\|_{L^2}^2 - \frac{1}{3} \|\nabla W\|_{L^2}^2 \quad &&\text{by (\ref{eq:Sobc})}\\
&= \frac{1}{2} \|\nabla W\|_{L^2}^2 + \frac{2}{3} \int_{\mathbb{R}^2} W \cdot (W_x \wedge W_y) dx \quad &&\text{by (\ref{eq:West2})}\\
&= \mathcal{E}(W)
\end{align*}
Note incidentally that this also shows 
\begin{equation} \label{eq:EW}
\mathcal{E}(W) = \frac{1}{6} \|\nabla W\|_{L^2}^2 > 0.
\end{equation}
The graph of $f$ is thus as follows:

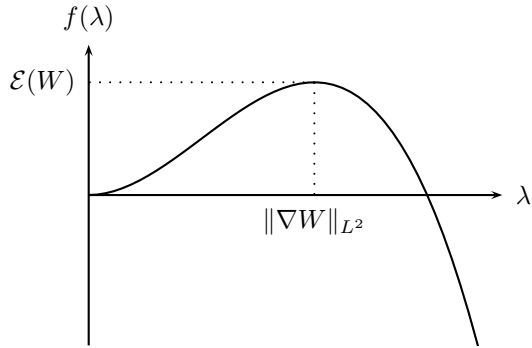
\begin{center}
 \begin{pspicture*}(-2,-2)(6,3)
   \psaxes[labels=none,ticks=none]{->}(0,0)(0,-2)(5.5,2)
   \psplot{0}{5.5}{x x mul 2 div x x mul x mul 9 div sub}
   \uput[180](0,1.5){$\mathcal{E}(W)$}          
   \uput[-90](3,0){$\|\nabla W\|_{L^2}$}
	\uput[0](5.5,0){$\lambda$}
	\uput[90](0,2){$f(\lambda)$}
   \psline[linestyle=dotted](0,1.5)(3,1.5) 
   \psline[linestyle=dotted](3,0)(3,1.5) 
 \end{pspicture*}
\end{center}

\begin{lem} \label{lem:gradugap}
For each $\delta > 0$, there exists $\bar{\delta} > 0$ such that if $u$ is a $\dot{H}^{1}$ function on $\mathbb{R}^2$ taking values in $\mathbb{R}^3$, with $\mathcal{E}(u) \leq (1-\delta) \mathcal{E}(W)$, then $$\|\nabla u\|_{L^2} > \|\nabla W\|_{L^2} + \bar{\delta} \quad \text{if $\|\nabla u\|_{L^2} > \|\nabla W\|_{L^2}$},$$ and $$\|\nabla u\|_{L^2} < \|\nabla W\|_{L^2} - \bar{\delta} \quad \text{if $\|\nabla u\|_{L^2} < \|\nabla W\|_{L^2}$}.$$
\end{lem}

\begin{proof}
The assumption $\mathcal{E}(u) \leq (1-\delta) \mathcal{E}(W)$ says $$f(\|\nabla u\|_{L^2}) \leq (1-\delta) f(\|\nabla W\|_{L^2});$$ this is because the Sobolev inequality (\ref{eq:Sob}) implies
\begin{align*}
f(\|\nabla u\|_{L^2}) 
&= \frac{1}{2} \|\nabla u\|_{L^2}^2 - \frac{2}{3} C \|\nabla u\|_{L^2}^3 \\
&\leq \frac{1}{2} \|\nabla u\|_{L^2}^2 + \frac{2}{3} \int_{\mathbb{R}^2} u \cdot (u_x \wedge u_y) dx \quad &&\text{by (\ref{eq:Sob})}\\
&= \mathcal{E}(u)
\end{align*}
which by our hypothesis is bounded by $(1-\delta) \mathcal{E}(W) = (1-\delta) f(\|\nabla W\|_{L^2})$. Hence by continuity of $f$, one can find some $\bar{\delta}$ such that $$\left|\|\nabla u\|_{L^2} - \|\nabla W\|_{L^2}\right| > \bar{\delta},$$ which implies the desired conclusion since $f$ is strictly increasing from $0$ to $\|\nabla W\|_{L^2}$, and strictly decreasing from $\|\nabla W\|_{L^2}$ to $+\infty$. (c.f. graph above.)
\end{proof}

\subsection{Finite time blow up}

An easy corollary of Lemma~\ref{lem:gradugap} is the following:

\begin{prop} \label{prop:gradugaptimet}
Under conditions of Theorem~\ref{thm:finitetimeblowup}, 
$$
\|\nabla u(t,x)\|_{L^2(dx)} > \|\nabla W\|_{L^2(dx)}
$$ 
for all $t \in [0,T)$.
\end{prop}

\begin{proof}
Recall that $\|\nabla u(t,x)\|_{L^2(dx)}$ is a continuous function of $t$, $E(W,0) = \mathcal{E}(W)$, and by energy conservation, for any time $t \in [0,T)$, we have $$E(u(t),u_t(t)) \leq (1-\delta) \mathcal{E}(W)$$ for some $\delta > 0$ independent of $t$. If $\|\nabla u(t,x)\|_{L^2(dx)} \leq \|\nabla W\|_{L^2(dx)}$ for some $t \in [0, T)$, let $t_0$ be the smallest value of $t$ that verifies this. Then $t_0 > 0$, and for any $t < t_0$, we have $\|\nabla u(t,x)\|_{L^2(dx)} \geq \|\nabla W\|_{L^2(dx)} + \bar{\delta}$ where $\bar{\delta} > 0$ is as in Lemma~\ref{lem:gradugap}. Letting $t \to t_0^-$, we arrive at a contradiction.
\end{proof}

\begin{proof}[Proof of Theorem~\ref{thm:finitetimeblowup}.]
Let $u$ be as in Theorem~\ref{thm:finitetimeblowup}, and let
$$
y(t) = \|u(t,x)\|_{L^2(dx)}^2
$$
for $t \in [0,T)$. Then 
$$
y'(t) = \int_{\mathbb{R}^2} 2 u \cdot u_t dx
$$
and
$$
y''(t) = \int_{\mathbb{R}^2} 2 |u_t|^2 - 2 |\nabla u|^2 - 4 u \cdot (u_x \wedge u_y) dx.
$$
These calculations are justified since $u$ is assumed to be smooth, and $u$ has compact support on each time slice. Suppose now $E(u_0,u_1) \leq E(W,0) - \varepsilon$ for some $\varepsilon > 0$. Then by conservation of energy,
$$
\int_{\mathbb{R}^2} -4 u \cdot (u_x \wedge u_y) dx \geq \int_{\mathbb{R}^2} 3 (|u_t|^2 + |\nabla u|^2) dx - 6 E(W,0) + 6 \varepsilon.
$$ 
This implies
\begin{align*}
y''(t) 
& \geq \int_{\mathbb{R}^2} (5 |u_t|^2 + |\nabla u|^2 - 3 |\nabla W|^2 - 4 W \cdot (W_x \wedge W_y)) dx + 6 \varepsilon \\
& = \int_{\mathbb{R}^2} (5 |u_t|^2 + |\nabla u|^2 - |\nabla W|^2) dx + 6 \varepsilon \qquad \text{by (\ref{eq:West2})} \\
& \geq 5 \|u_t(t,x)\|_{L^2(dx)}^2 + 6 \varepsilon \qquad \text{by Proposition~\ref{prop:gradugaptimet}}.
\end{align*}
In particular, $y'(t) > 0$ for all sufficiently large $t$, say for all $t > t_1$. Then for all $t > t_1$,
$$
y(t) y''(t) \geq \frac{5}{4} y'(t)^2,
$$
which implies (since $y'(t) > 0$) that
$$
\frac{y''(t)}{y'(t)} \geq \frac{5}{4} \frac{y'(t)}{y(t)}.
$$
It follows that
$$
(\log y'(t))' \geq \frac{5}{4} (\log y(t))',
$$
which implies that
$$
y'(t) \geq C y(t)^{5/4}
$$
for all $t > t_1$ where $C > 0$. Hence $y(t)$ becomes infinite in finite time, and therefore $T$ cannot be infinite.
\end{proof}

\bibliography{Liouville}

\end{document}